\documentclass[12pt,draft]{article}
\usepackage{amssymb,amsmath,amsthm}
\newtheorem{theorem}{Theorem}

\newtheorem{lemma}{Lemma}

\theoremstyle{definition}
\newtheorem{definition}{Definition}
\theoremstyle{remark}
\newtheorem{remark}{Remark}

\begin{document}
\title{Wa\'zewski Topological  Principle\\
and V-bounded Solutions of Nonlinear Systems}
\author{Volodymyr Lagoda\thanks{National Taras Shevchenko Univesity of
Kyiv, Volodymyrs'ka 64, Kyiv, 01033, Ukraine},\and  Igor
Parasyuk\thanks{ibid.}}
 \maketitle
\abstract{We use the Wa\'zewski topological principle to establish a number of
new sufficient conditions for the existence of proper (defined on the entire
time axis) solutions of essentially nonlinear nonautonomous systems. The
systems under consideration are characterized by the monotonicity property with
respect to a certain auxiliary  guiding function $W(t,x)$ depending on time and
phase coordinates. Another auxiliary function $V(t,x)$, which is positively
defined in the phase variables $x$ for any $t$, is used to estimate the
deviation of the  proper solutions from the origin.}

\section{Introduction}

 The goal of this paper is to lay down sufficient conditions
 under which the nonlinear nonautono\-mous system of ODEs
\begin{equation}\label{eq:nlsys} \dot{x}=f(t,x)
\end{equation}
where $f:\Omega  \mapsto \mathbb R^n$ ($\Omega\subset \mathbb R^{1+n}$)  has a solution $x(t)$
extendable on the entire time axis and  possessing the property that a given
positively definite (with respect to $x$-variables) function $V(t,x)$ is
bounded along the graph of $x(t)$. We especially focus on getting estimates for
the function $V(t,x(t))$. The main results are obtained by using the Wa\'zewski
topological principle \cite{Waz47,Har70,Con76,Con78}, and some of them
generalize the results of V.~M.~Cheresiz \cite{Che74}.

It should be noted that the Wa\'zewski topological principle was successfully
explo\-it\-ed for proving the existence of bounded solutions to some boundary
value problems in \cite{Con75} and to  quasihomogeneous systems in
\cite{Iva85a,Iva85b} (see also a discussion in \cite{Ort08}).

To apply the the Wa\'zewski principle, along with the function $V$ which can be
naturally considered as an analogue of time-dependent norm, we use another
auxiliary function $W(t,x)$. In general case, this function is a sign-changing
one, but it must have positively definite derivative by virtue of the system
\eqref{eq:nlsys} in the domain where $V\ge v_0$ for some constant $v_0>0$. We
call $V$ and $W$ \emph{the estimating function} and \emph{the guiding function}
respectively and we say that together they form the V--W-pair of the system.
Note that the term ''guiding function'' we borrow from \cite{KraZab75}
(originally --- ''guiding potential''). Basically topological method of guiding
functions, which was developed by M.~A.~Krasnosel'ski and A.~I.~Perov, is an
effective tool for proving the existence of bounded solutions of essentially
nonlinear systems too (see the bibliography in \cite{KraZab75, MawWar02}). But,
except \cite{Ort08,Avr03}, in all papers known to us only independent of time
guiding functions were used.

In \cite{Che74}, the role of V--W-pair plays  the square of Euclidean norm
together with an indefinite nondegenerate quadratic form. It appears that in
this case sufficient conditions for the existence of bounded solutions as well
as the estimates of their norms coincide with those obtained by means of
technique developed in \cite{PerTru74, PerTru86} for indefinitely monotone (not
necessarily finite dimensional) systems.

We shall not mention here another interesting approaches in studying the
existence problem of bounded solutions to nonlinear systems, because they have
not been used in this paper. For the corresponding information the reader is
referred to
\cite{War94,SBB02,Blo88,BerZha95,BCM97,ZakPar99,Cie03,CheMam05,Her06,
ARS06,Sly02,Gro06,Kar07}.

This paper is organized as follows. Section~\ref{defVWpair} contains necessary
definitions, in particular, the notion of V--W-pair is introduced and some
additional conditions imposed on estimating and guiding functions are
described. In section~\ref{ExistUniqVBS} we prove two main theorems concerning
the existence and the uniqueness of V-bounded solution to a nonlinear
nonautonomous system possessing V--W-pair. Finally, in
section~\ref{PairQuadForms} we show how the results of
section~\ref{ExistUniqVBS} can be applied in the case where the estimating and
guiding functions are nonautonomous quadratic forms. In this connection it
should be pointed out that guiding quadratic forms play an important role in
the theory of linear dichotomous systems with (integrally) bounded coefficients
\cite{DalKre70,MSK03,Sam02}.

\section{The definition of V--W-pair
 and the main\\ assumptions} \label{defVWpair}

Let  $\Omega$ be a domain of $\mathbb R^{1+n}=\{t\in \mathbb R\}\times \{x\in \mathbb R^{n}\}$ such that
the projection of $\Omega $ on the time axis $\{t\in \mathbb R\}$ covers all this axis,
and let in the system \eqref{eq:nlsys}  $f(\cdot)\in C(\Omega \mapsto \mathbb R^n)$. It will be
always assumed that each solution of the system has the uniqueness property.

\begin{definition}A function $V(\cdot)\in  C^1(\mathbb R\times \mathbb R^n\! \mapsto\! \mathbb R_+)$
of variables $t\in \mathbb R,\;x\in \mathbb R^{n}$ will be called the estimating function, if
for any $t\in \mathbb R$ the function $V_t(\cdot):=V(t,\cdot):\mathbb R^n \mapsto \mathbb R_+$ is positively
definite, has a unique critical point, the origin, and satisfies the condition
$\lim_{\|x\|\to \infty}V_t(x)= \infty$.
\end{definition}

Note that, as is well known, for any $t\in \mathbb R$ and each $c>0$ the set
$V_t^{-1}([0,c]):=\{x\in \mathbb R^{n}:0\le V_t(x)\le c\}$ is compact, its boundary is
a closed connected hypersurface $V_{t}^{-1}(c)$ surrounding the origin, and in
addition, if  $c_2>c_1$, then the set $V_t^{-1}([0,c_1])$ is a proper subset of
the set $V_t^{-1}([0,c_2])$.

\begin{definition}A global solution $x(t),\;t\in I$ of the system \eqref{eq:nlsys}
is said to be V-bounded if  $\sup_{t\in I}V(t,x(t))<\infty $.
\end{definition}

For  $U(\cdot)\in C^1(\Omega  \mapsto \mathbb R)$ we put $$\dot{U}_{f}:=\frac{\partial V}{\partial t}+\frac{\partial
V}{\partial x}\cdot f.$$

\begin{definition}For the system \eqref{eq:nlsys}, a function $W(\cdot)\in C^1(\Omega \! \mapsto\! \mathbb R)$
will be called the guiding function concordant with $V$ if for some $v_0>0$
such that $\Omega \cap V^{-1}\bigr([v_0,\infty)\bigl)\ne \varnothing$ there exist
functions
\begin{gather*}
  a(\cdot)\in C\bigl(\Omega \cap V^{-1}\bigl([v_0,\infty)\bigr)\! \mapsto\!(0,\infty)\bigr),\\
G(\cdot)\in C([v_0,\infty)\! \mapsto\!(0,\infty)), \quad g(\cdot)\in C([v_0,\infty)\! \mapsto\! (0,\infty))
\end{gather*}
satisfying the inequalities
\begin{gather}
  \label{eq:imcond1}
\bigl|\dot{V}_{f}(t,x)\bigr|  \le a(t,x)G(V(t,x))\quad \forall(t,x)\in \Omega \cap
V^{-1}\bigl([v_0,\infty)\bigr),\\ \label{eq:imcond2}
 \dot{W}_{f}(t,x) \ge a(t,x)g(V(t,x))\quad \forall(t,x)\in
 \Omega \cap V^{-1}([v_0,\infty)),\\ \label{eq:imcond3}
g(v)\ge g(v_0)>0\quad \forall v\ge v_0.
\end{gather}
\end{definition}

\begin{definition}For the system \eqref{eq:nlsys}, the estimating function $V$
and the concordant  guiding function $W$
 will be called the V--W-pair of this system.
\end{definition}

Define
\begin{gather*}
  F(v):=\int_{v_0}^{v}\big(g(u)/G(u)\big)\,du.
\end{gather*}
On the half-line $v\ge v_0$, this function is monotonically increasing and has
the inverse $F^{-1}(\cdot):[0,\infty) \mapsto[v_0,\infty)$.

Denote by $\Pi_t:=\{t\}\times \mathbb R^{n}$ the ''vertical'' hyperplane in $\mathbb R^{1+n}$,
and in so far suppose that the system \eqref{eq:nlsys} has V--W-pair which
satisfies the following additional conditions:
\newlength{\ppp}\setlength{\ppp}{\textwidth}
\addtolength{\ppp}{-6.\parindent}

\medskip

 \noindent(A):\quad \parbox[t]{\ppp}{$\lim_{v\to \infty}F(v)=\infty $;}

 \noindent(B):\quad
 \parbox[t]{\ppp}{$\int_{-\infty}^{0}\alpha (s)\,ds=\int_{0}^{\infty}\alpha (s)\,ds=\infty$,
 where
$$
  \alpha(t):=\inf \left\{a(t,x): x\in \Omega_t, \;V_t(x)>v_0\right\};
$$}

 \noindent(C):\quad \parbox[t]{\ppp}{there exist numbers
 $w^+,\;w_-$ ($w^+>w_-$) such that
 $$V^{-1}\bigl([0,v_0)\bigr)\subset W^{-1}\bigl((w_-,w^+)\bigr),\quad
V^{-1}(v_0)\subset \Omega,$$ and in addition, for any $t\in \mathbb R$ the number $w^+$
belongs to the range of  $W_{t}(\cdot):=W(t,\cdot):\Omega_t \mapsto \mathbb R$ where
$\Omega_{t}:=\Pi_t\cap \Omega $.}

\noindent(D):\quad
\parbox[t]{\ppp}{the domain $\mathcal{W}$, which is defined as such a connected component of the set
$W^{-1}\bigl(w_-,w^+)$ that contains $V^{-1}\bigl([0,v_0)\bigr)$, has the
property:  for any sufficiently large by absolute value negative  $t$ there
exists a set $\mathcal{M}_{t}\subset \mathcal{W}_t\cup[\partial \mathcal{W}_t\cap W_t^{-1}(w^+)]$, where
$\mathcal{W}_t:=\mathcal{W}\cap \Pi_t$, such  that the set  $ \mathcal{M}_{t}\cap \partial \mathcal{W}_t\cap
W_t^{-1}(w^+)\ne\varnothing$ is a retract of $\cup_{s\ge t}\partial \mathcal{W}_s\cap
W_s^{-1}(w^+)$, but is not a retract of $\mathcal{M}_t$, and, besides,
\begin{gather*}
 \liminf_{t\to -\infty}\sup\{V_t(x):x\in
\mathcal{M}_t\}=\nu <\infty.
\end{gather*}}

\begin{remark}\label{rem:1}If $V^{-1}\bigl([0,v_0]\bigr)
\subset \Omega $ and $W(V^{-1}(v_0))\in [w_-,w^+]$, then one can redefine the
guiding function in the domain $V^{-1}\bigl([0,v_0)\bigr)$ in such a way that
$V^{-1}\bigl([0,v_0)\bigr)\subset W^{-1}\bigl((w_-,w^+)\bigr)$.
\end{remark}

\begin{remark} The condition (D) is fulfilled if for any negative
sufficiently large by abso\-lute value  $t$ there exists a finite collection
$\{\mathcal{M}_{t,j}\}$ of compact manifolds with border such that:  $\partial
\mathcal{M}_{t,j}\cap \partial \mathcal{M}_{t,k}=\varnothing $, $j\ne k$; the interior of
$\mathcal{M}_{t,j}$ belongs to $\mathcal{W}_t$; the set $\cup_{j\ge 1}\partial \mathcal{M}_{t,j}$ is a
retract of $\cup_{s\ge t}\partial \mathcal{W}_s\cap W_s^{-1}(w^+)$ and
\begin{gather*}
  \liminf_{t\to -\infty}\max\{V_t(x):x\in
\cup_{j\ge 1}\mathcal{M}_{t,j}\}=\nu <\infty.
\end{gather*}
In fact, in this case, taking into account that any compact manifold can not be
retracted to its border,  it is sufficient to put $\mathcal{M}_t:=\cup_{j\ge
1}\mathcal{M}_{t,j}$.
\end{remark}

\begin{remark}\label{rem:(D)} Since the set
$\mathcal{M}_{t}\cap W_t^{-1}(w^+)$ is not empty and in any point of this set the
function $V$ takes values not less than  $v_0$,  we get the inequality $\nu \ge
v_0$.\end{remark}

\section{The existence and the uniqueness\\
 of V-bounded solution}\label{ExistUniqVBS}

The lemma given below open the door to estimation of solutions of the system
\eqref{eq:nlsys} by means of functions $V$ in the presence of V--W-pair.

\begin{lemma}\label{lem:estimate}
Suppose that the system \eqref{eq:nlsys} has V--W-pair satisfying the condition
(A). Let this system has a global solution $x(t)$, $t\in I\subseteq \mathbb R$,  such
that
\begin{gather*}
  W^*:=\sup_{t\in J} W(t,x(t))<\infty,\quad
  W_*:=\inf_{t\in J} W(t,x(t))>-\infty
\end{gather*}
äå $J:=\{t\in I:V(t,x(t))> v_0\}$.

Then for any $t_0\in I$, in the case where  $V(t_0,x(t_0))\le v_0$ and
$J\ne\varnothing$, the following inequality holds true
\begin{gather}\label{eq:estonI1}
  V(t,x(t))\le F^{-1}\big(W^*-W_0\big)\quad \forall t\in I\cap[t_0,\infty)
\end{gather}
where
\begin{gather*}
  W_0=\inf \left\{W(t,x(t)):\;t\in I,\;V(t,x(t))=v_0\right\}.
\end{gather*}

If $V(t_0,x(t_0))> v_0$, then in the case where $V(t,x(t))>v_0$ for all $t\in
I\cap [t_0,\infty)$, we have
\begin{gather}\label{eq:estonI2}
  V(t,x(t))\le F^{-1}\big(F\big(V(t_0,x(t_0))\big)+W^*-W(t_0,x(t_0))\big),
\end{gather}
and otherwise
\begin{gather}\label{eq:estonI3}
  V(t,x(t))\le
    \max \left\{F^{-1}\big(F\big(V(t_0,x(t_0))\big)+{W}^0-W(t_0,x(t_0))\big),
     F^{-1}\big(W^*-W_0\big)\right\},
     \end{gather}
where
\begin{gather*}
  W^0=\sup \left\{W(t,x(t)):\;t\in I,\;V(t,x(t))=v_0\right\}.
\end{gather*}
In addition, if $[t_*,t^*]\subset I$ is such a segment that $(t_*,t^*)\subset
J$ and $$V\big(t_*,x(t_*)\big)=V\big(t^*,x(t^*)\big)=v_0,$$ then
\begin{gather}\label{eq:estV}
  V(t,x(t))\le F^{-1}\left(\frac{1}{2}\Big[W\big(t^*,x(t^*)\big)-W\big(t_*,x(t_*)\big)\Big]\!\right)
  \quad \forall t\in[t_*,t^*].
\end{gather}

If the condition (B) is fulfilled and $[t_0,\theta]\subset I$, where  $\theta $ is
determined by the equality
\begin{gather} \label{eq:esttau}
  \int_{t_0}^{\theta}\alpha (s)\,ds=(W^*-W_*)/g(v_0),
\end{gather}
then there exists  $\tau \in  [t_0,\theta]$ for which $V(\tau,x(\tau))\le v_0$.
\end{lemma}
\begin{proof} Let the condition (A) is fulfilled.
Throughout this proof, put $v(t):=V(t,x(t))$. Then in view of
\eqref{eq:imcond1}, \eqref{eq:imcond2} we have
\begin{gather}\label{eq:vW}
 \left|\frac{d}{dt}F(v(t))\right|=  \frac{g(v(t))|\dot{v}(t)|}{G(v(t))}\le
 \frac{d}{dt}W(t,x(t))\quad \forall t\in J.
\end{gather}

If $v(t_0)\le v_0$ and $J\ne\varnothing$, then there exists an interval
$(t_*,T)\subset J\cap [t_0,\infty)$ such that  $v(t_*)=v_0$. Since $F(v(t_*))=0$,
then, as a consequence of \eqref{eq:vW} and inequality $W(t_*,x(t_*))\ge W_0$,
we have
\begin{gather*}
  F(v(t))\le W(t,x(t))-W(t_*,x(t_*))\le W^*-W_0 \quad
  \forall t\in [t_*,T],
\end{gather*}
and from this it follows that $v(t)\le F^{-1}(W^*-W_0)$ for all $t\in [t_*,T]$.
Taking into account that $v_0=F^{-1}(0)\le F^{-1}(W^*-W_0)$ and the function
$F^{-1}(\cdot)$ is monotonically increasing, one ascertains that \eqref{eq:estonI1}
is true for all $t\in [t_0,\infty)\cap I$.

If now  $v(t_0)>v_0$, then until $v(t)>v_0$ we have
\begin{gather*}
  F(v(t))-F(v(t_0))\le W(t,x(t))-W(t_0,x(t_0)).
\end{gather*}
In the case where $v(t)>v_0$ for all $t\in [t_0,\infty)\cap I$, we obtain the
inequality \eqref{eq:estonI2}. Otherwise there exists the nearest to $t_0$
moment $t_*>t_0$ such that $v(t_*)=0$. Then $W(t_*,x(t_*))\le W^0$, and on the
segment $[t_0,t_*]$, we arrive at
\begin{gather*}
  F(v(t))-F(v(t_0))\le
W^0-W(t_0,x(t_0))\quad \Leftrightarrow\\ v(t)\le
F^{-1}\big(F(v(t_0))+W^0-W(t_0,x(t_0))\big).
\end{gather*}
 Taking into account the estimate
obtained above for $v(t)$ on the segment $[t_*,T]$,  one ascertains that the
inequality \eqref{eq:estonI3} holds true.

Now let us estimate $v(t)$ on $[t_*,t^*]\subset I$ under the condition that
$(t_*,t^*)\subset J$ and $v(t_*)=v(t^*)=v_0$. Let $\hat{t}$ be a point at which
$v(t)$ reaches its maximum on $[t_*,t^*]$. Then from the inequality
\eqref{eq:vW} it follows that
\begin{gather*}
 W(t^*,x(t^*))-W(t_*,x(t_*))\ge
 \int_{t_*}^{\hat{t}}\frac{g(v(t))|\dot{v}(t)|}{G(v(t))}\,dt+
\int_{\hat{t}}^{t^*}\frac{g(v(t))|\dot{v}(t)|}{G(v(t))}\,dt\ge \\
2F(v(\hat{t}))-F(v(t_*))-F(v(t^*))=2F(v(\hat{t}))\ge 2F(v(t))\quad \forall t\in
[t_*,t^*],
\end{gather*}
and  we obtain the inequality \eqref{eq:estV}.

Next, let the condition (B) is fulfilled and $[t_0,\theta]\subset I$. Let us prove
that there exists a number  $\tau $ belonging to $[t_0,\theta]$ for which $v(\tau)\le
v_0$. Obviously, it is sufficient to consider the case where $v(t_0)>v_0$. If
we suppose the contrary, i.e. that $v(t)>v_0$ for all $t\in [t_0,\theta]$, then we
can find such a small $\epsilon>0$ that the inequality $v(t)>v_0$ and thus the
inequality $\frac{d}{dt}W(t,x(t))\ge \alpha (t)g(v_0)>0$ holds true for all $t\in
[t_0,\theta +\epsilon]\subset J$. From this in virtue of ~\eqref{eq:imcond3} we arrive
at inequality $$W^*\ge W(\theta +\epsilon,x(\theta +\epsilon))> g(v_0)\int_{t_0}^{\theta}\alpha (s)\,ds
+ W(t_0,x(t_0))\ge g(v_0)\int_{t_0}^{\theta}\alpha (s)\,ds+W_*$$ which contradicts the
definition of $\theta$. Hence, there do exists a number $\tau \in [t_0,\theta]$ with
the required property.
\end{proof}

\begin{remark}\label{rem:4}
If it is impossible to find $F^{-1}(\cdot)$ explicitly, then in order to obtain
efficient estimates of solutions one can replace the function $F(v)$ by another
appropriate strictly monotonic function $F_1(v)$, which satisfies the
inequality $F_1(v)\le F(v)$ for all $v>v_0$ and tends to infinity when $v\to \infty
$.
\end{remark}

Put
\begin{gather*}
  w^0(t):=\max\left\{W_t(x):x\in V_t^{-1}(v_0)\right\},\\
  w_0(t):=\min\left\{ W_t(x):x\in V_t^{-1}(v_0)\right\},\\
   \omega_0:=\liminf_{t\to -\infty}w_0(t),\\
  \tilde{\omega}:=\liminf_{t\to - \infty}\bigl(\inf\{W_t(x):x\in \mathcal{M}_t,\;V_t(x)\ge
  v_0\}\bigr).
\end{gather*}
It is clear that the inequalities   $$w_-\le \omega_0\le w^+,\quad w_-\le\tilde{\omega}\le
w^+$$ holds true once the condition (D) is satisfied.

Now we are in position to prove the following statement.

\begin{theorem}\label{th:exbs} Assume that the system \eqref{eq:nlsys}
has V--W-pair satisfying the conditions (A)-(D). Let there exists a number
$V^*$ such that
\begin{equation}\label{eq:V*>}
  V^*>\max\bigl\{F^{-1}\bigl(F(\nu)+w^+-\tilde{\omega}\bigr),F^{-1}(w^+-
\omega_0)\bigr\},
\end{equation}
and the set $\mathrm{cls}\left(V^{-1}\bigl([0,V^*)\bigr)\cap\mathcal{W}\right)$ (here $\mathrm{cls}$
means the closure operation) belongs to the domain $\Omega $. Then the system
\eqref{eq:nlsys} has a V-bounded solution $x_*(t),\;t\in \mathbb R$, which satisfies
the inequality
\begin{gather}\label{eq:estVfin}
  V(t,x(t))\le F^{-1}\left(\frac{1}{2}\Big[\sup_{s\ge t}w^0(s)-\inf_{s\le t}w_0(s)\Big]\right)\le
  F^{-1} \left(\frac{w^+ -w_-}{2}\right)=:v_* \quad
  \forall  t\in \mathbb R.
\end{gather}
\end{theorem}
\begin{proof}
By the conditions (C) and (D) the set $\partial \mathcal{W}\cap W^{-1}(w^+)$ does not
intersect  the set $V^{-1}([0,v_0))$. Then from the definition of guiding
function it follows that the set $\partial \mathcal{W}\cap W^{-1}(w^+)$ coincides with the
set of exit points of integral curves of the system  \eqref{eq:nlsys} from the
domain $\mathcal{W}$ and it consists  of the strict exit points only.

By the condition (D) we can choose a sequence of moments $t_j\to -\infty,\;j\to \infty
$, and a sequence  of sets  $\mathcal{M}_{t_j} \subset \mathcal{W}_{t_j}\cup[\partial
\mathcal{W}_{t_j}\cap W_{t_j}^{-1}(w^+)]$ in such a way that
\begin{gather*}
\sup\{V_{t_j}(x):x\in \mathcal{M}_{t_j}\}\le \nu +\delta,\quad
 \inf\{W_{t_j}(x):x\in \mathcal{M}_{t_j},\;V_{t_j}(x)\ge
  v_0\}\ge \tilde{\omega}-\delta, \\
V^*> \max\{F^{-1}(F(\nu +\delta)+w^+-\tilde{\omega}+\delta),
 F^{-1}(w^+-\omega_0)\}\ge \nu +\delta >v_0
\end{gather*}
for sufficiently small $\delta>0$ and for all $j$, and the intersection of each
$\mathcal{M}_{t_j}$ with $\partial \mathcal{W}\cap W^{-1}(w^+)$ be the retract for the set of exit
points from  $\mathcal{W}\cap \{(t,x):t\ge t_j\}$ but there does not exist a
retraction of $\mathcal{M}_{t_j}$ on $\mathcal{M}_{t_j}\cap \partial \mathcal{W}\cap W^{-1}(w^+)$. Then by
Wa\'zewski principle for any $j$ there exists a point  $(t_j,x_{0j})\in
\mathcal{M}_{t_j}$ such that the nonextendable  solution $x_j(t),\;t\in I_j$, which
satisfies the initial condition $x_j(t_j)=x_{0j}$ has the property
\begin{gather*}
  (t,x_j(t))\in \mathcal{W}  \quad
  \forall t\in [t_j,\infty)\cap I_j.
\end{gather*}
Observe that $V(t_j,x_{0j})\le \nu +\delta $, and thus by the
lemma~\ref{lem:estimate} setting $I=[t_j,\infty)\cap I_j$, $v_j(t)=V(t,x_j(t))$ we
obtain $v_j(t)<V^*,\;t\in I$. Hence, taking into account the condition of the
theorem we have
\begin{gather*}
 (t,x_j(t))\in
 \mathrm{cls}\left(V^{-1}\bigl([0,V^*)\bigr)\cap\mathcal{W}\right)\subset \Omega,\quad t\in [t_j,\infty)\cap I_j.
\end{gather*}
In view of this  we conclude that $[t_j,\infty)\subset I_j$.

Next, applying the lemma~\ref{lem:estimate} again, we can find $\tau_j$ for which
$v_j(\tau_j)\le v_0$, and from (B) it follows that $\tau_j\to -\infty,\;j\to -\infty $.
Besides, if there exists at least one $t\ge \tau_j$ for which $v_j(t)>v_0$, then
there exist moments $t_*$, $t^*$ such that $\tau_j\le t_*<t$, $t<t_*$ and
$v_j(t_*)=v_j(t^*)=v_0$, but $v_j(t)>v_0$ for $t\in (t_*,t^*)$. Then in virtue
of inequality \eqref{eq:estV}, for any pair of such moments we have
\begin{gather} \label{eq:estvj}
  v_j(t)\le F^{-1} \left(\frac{w^0(t^*)-w_0(t_*)}{2}\right)\quad
  \forall t\in [t_*,t^*],
\end{gather}
and thus $v_j(t)\le v_*$ for all $t\ge \tau_j$.

Now one can prove the existence of V-bounded solution $x_*(t)$ by the known
scheme (see, e.g., \cite{Che74,Iva85b,KraZab75}. Namely, if we denote by
$x(t,t_0,x_0)$  the solution  which for $t=t_0$ takes the value  $x_0$, then
setting $\xi_j:=x_j(0)$, we obtain the equalities
\begin{gather*}
  x_j(t)=x(t,0,x_j(0))=x(t,0,\xi_j),\quad t\in [t_j,\infty).
\end{gather*}
Having selected from the sequence $\xi_j \in
\mathrm{cls}\,\left(V_0^{-1}([0,v_*])\cap\mathcal{W}_0\right)\subset \Omega_0 $ a subsequence
converging to $x_*$, put $x_*(t):=x(t,0,x_*)$.  Using the reductio ad absurdum
reasoning it is easy to show that on the maximal existence interval $I$ of this
solution we have the inclusion $$(t,x_*(t))\in \mathrm{cls}(V^{-1}([0,v_*])\cap
\mathcal{W}).$$ Therefore  $I=\mathbb R$ and $V(t,x_*(t))\le v_*$ for all $t\in \mathbb R$. Finally,
taking into account  \eqref{eq:estvj}, we arrive at the inequality
\eqref{eq:estVfin}.
\end{proof}

\begin{remark}\label{rem:5} As is easily seen from the proof of the Theorem~\ref{th:exbs},
it is sufficient to require that the inequalities
\eqref{eq:imcond1},\eqref{eq:imcond2} hold true  on the set
 $\mathrm{cls}\left(V^{-1}\bigl([v_0,V^*)\bigr)\cap\mathcal{W}\right)$ only, and the inequalities
\eqref{eq:imcond3} --- for $v\in [v_0,V^*]$ only. \end{remark}

\begin{theorem}\label{th:unbs} Let  $\tilde{\Omega}$ be a subdomain of the domain $\Omega $
and let $$\Omega^*:=\{(t,z)\in \mathbb R\times \mathbb R^{n}:z=x-y,\;(t,x)\in \tilde{\Omega},\;(t,y)\in
\tilde{\Omega} \}.$$ Suppose that there exist functions $U(\cdot)\in C^1(\Omega^* \mapsto \mathbb R)$,
$H(\cdot),h(\cdot)\in C(\mathbb R_+ \!\!\mapsto\!\! \mathbb R_+)$, $b(\cdot),\beta(\cdot)\in C(\mathbb R \mapsto(0,\infty))$ such
that:

1) the function $h(\cdot)$ is nondecreasing, the function $H(\cdot)$ is strictly
monotonically increasing, and in addition,
\begin{gather*}
  \limsup_{t\to \pm \infty}\frac{1}{b(t)}\left|\int_{0}^{t}\beta (t)h\circ H^{-1} \left(\frac{u}{b(t)}\right)\,dt\right|
  =\infty
\end{gather*}
for any $u>0$;

2) for all $(t,x),(t,y)\in \tilde{\Omega}$, the following inequalities hold true
\begin{gather*}
|U(t,x-y)|\le
 b(t) H(V(t,x-y)),\\
  U'_t(t,x-y)+U'_x(t,x-y)\cdot(f(t,x)-f(t,y))\ge \beta (t)h(V(t,x-y)),\quad
\end{gather*}
where $V(t,x)$ is an estimating function. Then the system \eqref{eq:nlsys}
cannot have two different nonextendable solutions $x(t)$, $y(t)$, $t\in \mathbb R$,
whose graphs lie in  $\tilde{\Omega}$ and which have the property
\begin{equation}\label{eq:supVx-y}
  \sup_{t\in \mathbb R}
V(t,x(t)-y(t))<\infty.
\end{equation}
\end{theorem}
\begin{proof} Suppose that the system \eqref{eq:nlsys} has a pair
of solutions $x(t)$, $y(t),\;t\in \mathbb R$ such that $(t,x(t))$, $(t,y(t))\in
\tilde{\Omega}$ for all $t\in \mathbb R$. Let us show that for these solutions the condition
\eqref{eq:supVx-y} fails.

Consider the functions $u(t):=U(t,x(t)-y(t)),\;v(t):=V(t,x(t)-y(t))$. By
condition, the function $u(\cdot)$ does not decrease. Hence, there exist (either
finite or infinite) limits $u_*=\lim_{t\to -\infty}u(t)$, $u^*=\lim_{t\to \infty}u(t)$.
If we suppose that $x(t)\not\equiv y(t)$, then there exists $t_0$ such that
$x(t_0)\ne y(t_0)$, from whence $v(t_0)>0$ and $\dot{u}(t_0)>0$. For this reason,
$u^*> u(t_0)> u_*$.

First suppose  that  $u_*\ge 0$. Then $u(t)> u(t_0)>0$ for $t> t_0$. Since
$$v(t)\ge H^{-1}(u(t_0)/b(t)),$$ then
\begin{gather*}
  u(t)\ge u(t_0)+ \int_{t_0}^{t}\beta (s)h\circ H^{-1} \left(\frac{u(t_0)}{b(s)}\right)\,ds
\end{gather*}
and
\begin{gather*}
  H(v(t))\ge \frac{1}{b(t)}\int_{t_0}^{t}\beta (s)h\circ H^{-1}
  \left(\frac{u(t_0)}{b(s)}\right)\,ds,\quad t\ge t_0.
\end{gather*}
Thus, $\limsup_{t\to \infty}v(t)=\infty $.

Now suppose  that  $u_*<0$. Then there exists  $t'$ such that $u(t')<0$. Then
$u(t)\le u(t')$ for all $t<t'$ and $v(t)\ge H^{-1}(|u(t')|/b(t))$ for $t<t'$.
Then
\begin{gather*}
  u(t')-u(t)\ge \int_{t}^{t'}\beta (s)h\circ H^{-1}
  \left(\frac{|u(t')|}{b(s)}\right)\,ds,\quad t\le t',
\end{gather*}
from whence, as above, we have $\limsup_{t\to -\infty}v(t)=\infty $.
\end{proof}

\section{Studying V-bounded solutions by means of pair of quadratic forms}
\label{PairQuadForms}

Consider the case where the V--W-pair of the system \eqref{eq:nlsys} is a pair
of  quadratic forms
\begin{gather}\label{eq:VWbilin}
  V(t,x)=\langle B(t)x,x\rangle,\quad W(t,x)=\langle C(t)x,x\rangle,
\end{gather}
where $\langle \cdot,\cdot \rangle $ is a scalar product in $\mathbb R^{n}$, $\left\{B(t)\right\}_{t\in \mathbb R}$
and $\left\{C(t)\right\}_{t\in \mathbb R}$ are families of symmetric nondegenerate operators
in $\mathbb R^{n}$ smoothly depending on parameter  $t$ and satisfying the conditions:

\medskip

\noindent (a):\quad   for any $t\in \mathbb R$, the operator $B(t)$ is positively
definite;

\noindent (b):\quad \parbox[t]{\ppp}{for any $t\in \mathbb R$, there exist projectors
$P_+(t), P_-(t)$ on corresponding invariant subspaces $\mathbb L_{+}(t),\mathbb L_{-}(t)$ of
operator $C(t)$ such that the restriction of $C(t)$ on $\mathbb L_+(t)$ (on $\mathbb L_-(t)$)
is a positively definite (negatively definite) operator.}

\medskip

Observe that since the subspaces $\mathbb L_+(t),\;\mathbb L_-(t)$ are mutually orthogonal
then the projectors  $P_+(t),\;P_-(t)$ are symmetric.

From $C(t)$-invariance of these subspaces it follows that
$P_{\pm}(t)C(t)=C(t)P_{\pm}(t)$ and, as a consequence, we have the representation
\begin{gather*}
  C(t)=(P_+(t)+P_-(t))C(t)(P_+(t) +P_-(t))=P_+(t)C(t)P_+(t) +P_-(t)C(t)P_-(t).
\end{gather*}
Put
\begin{equation}\label{eq:C+C_}
  C_+(t):=P_+(t)C(t)P_+(t),\;C_-(t)=P_-(t)C(t)P_-(t)
\end{equation}
Obviously, the kernel of the operator $C_+(t)$ (operator $C_-(t)$) is the
subspace  $\mathbb L_-(t)$ (subspace $\mathbb L_+(t)$), and the restriction of this operator
on $\mathbb L_+(t)$ (on $\mathbb L_-(t)$) is a positively definite (negatively definite)
operator.

Let the right-hand side of the system \eqref{eq:nlsys} admits the
representation
\begin{gather*}
  f(t,x)=A(t,x)x+f_0(t)
\end{gather*}
where $A(\cdot,\cdot)\in C(\Omega \mapsto \mathrm{Hom}\, \mathbb R^n),\;f_0(t):=f(t,0)$, and in addition,

\medskip

\noindent(c):\quad  \parbox[t]{\ppp}{there exist numbers $V^*>0$, $w_- <0$,
$w^+>0$ such that the domain $\Omega $ contains the set
\begin{gather*}
  V^{-1}\bigl([0,V^*]\bigr)\cap W^{-1}\bigl([w_-,w^+]\bigr)=\\
  \bigr\{(t,x)\in \mathbb R^{1+n}:\langle B(t)x,x\rangle \le V^*,\;w_-\le \langle C(t)x,x\rangle \le w^+\bigl\}.
\end{gather*}}

\medskip

The number  $v_0$ in the definition of the guiding function must be chosen in
such a way that the set inclusions from condition (C) hold true. Denote by $\lambda
=\lambda^+(t)$ and $\lambda =\lambda_-(t)$, respectively, the maximal and the minimal
characteristic values of the pencil $C(t)-\lambda B(t)$. Since
\begin{gather*}
  \lambda^+(t)=\max\big\{\langle C(t)x,x\rangle :\langle B(t)x,x\rangle =1\big\},\\
  \lambda_-(t)=\min\big\{\langle C(t)x,x\rangle :\langle B(t)x,x\rangle =1\big\}
\end{gather*}
(see, e.g., \cite{Gan67}), then taking into account that the function $W_t(x)$
has the unique critical point $x=0$, we have
\begin{gather*}
  w^0(t):=\max\big\{\langle C(t)x,x\rangle :\langle B(t)x,x\rangle \le
  v_0\big\}=\lambda^+(t)v_0,\\
w_0(t):=\min\big\{\langle C(t)x,x\rangle :\langle B(t)x,x\rangle \le v_0\big\}=\lambda_-(t)v_0.
\end{gather*}
Hence, in order that the set inclusions from condition (C) hold true it is
sufficient to assume that

\medskip

\noindent (d):\quad \parbox[t]{\ppp}{the inequalities
\begin{gather*}
\lambda_-(t)v_0\ge w_-,\quad \lambda^+(t)v_0\le w^+ \quad \forall t\in \mathbb R.
\end{gather*}
are fulfilled}

\medskip

Now we impose a number of conditions on the mappings  $A(t,x)$ and $f_0(t)$ to
ensure the existence of V-bounded solution of the system \eqref{eq:nlsys} in
virtue of the Theorem~\ref{th:exbs} and the Remark~\ref{rem:5}.

Let $\Lambda_V(t,x)$ be the maximal by absolute value  characteristic value of the
pencil $$B(t)A(t,x)+A^*(t,x)B(t)+\dot{B}(t)-\lambda B(t)$$ (here $A^*$ is the operator
conjugate with $A$), and let $\lambda_W(t,x)$ be the minimal characte\-ris\-tic
value of the pencil
 $$C(t)A(t,x)+A^*(t,x)C(t)+\dot{C}(t)-\lambda
B(t).$$ Put
\begin{gather*}
  \varphi(t):=\sqrt{\langle B(t)f_0(t),f_0(t)\rangle},\quad
  \psi(t):=\sqrt{\langle B^{-1}(t)C(t)f_0(t),C(t)f_0(t)\rangle}.
\end{gather*}
Then taking into account the inequalities
\begin{gather*}
  \left|\langle\big(2 B(t)A(t,x)+\dot{B}(t)\big)x,x\rangle \right|\le \left|\Lambda_V(t,x)\right|\langle B(t)x,x\rangle,\\
  \langle\big( 2C(t)A(t,x)+\dot{C}(t)\big)x,x\rangle \ge \lambda_W(t,x)\langle B(t)x,x\rangle,\\
   \langle B(t)f_0(t),x\rangle \le \varphi(t)\sqrt{\langle B(t)x,x\rangle},\\
   \langle C(t)f_0(t),x\rangle \ge
  -\psi(t)\sqrt{\langle B(t)x,x\rangle}
\end{gather*}
we obtain
\begin{gather*}
  \left|\dot{V}_{f(t,x)}(t,x)\right|\le \left|\Lambda_V(t,x)\right|V(t,x)+2\varphi(t)\sqrt{V(t,x)},\\
   \dot{W}_{f(t,x)}(t,x)\ge \lambda_W(t,x)V(t,x)-2\psi(t)\sqrt{V(t,x)}.
\end{gather*}
Now impose  on the system the following conditions:

\medskip

\noindent (e):\quad \parbox[t]{\ppp}{there exist positive constants
$c_1,\;c_2,\;c_3,\;\sigma $ such that
\begin{gather*}
  c_2^2 <v_0,\quad \sigma \le 1,
\end{gather*}
and in the domain  $V^{-1}\bigl([v_0,V^*]\bigr)\cap
W^{-1}\bigl([w_-,w^+]\bigr)$ the inequalities}
\begin{gather*}
2\varphi(t)\le c_1\left|\Lambda_V(t,x)\right|,\quad 2\psi (t)\le c_2\lambda_W(t,x),\quad
  \left|\Lambda_V(t,x)\right|\le c_3\langle B(t)x,x\rangle^{\sigma}\lambda_W(t,x);
\end{gather*}
\phantom{(e):\quad} holds true;

\noindent (f): \quad \parbox[t]{\ppp}{the function $\alpha(t):=\inf \left\{\lambda_W(t,x):
x\in \Omega_t, \;V_t(x)>v_0\right\}$ has the properties $$\int_{-\infty}^{0}\alpha
(s)\,ds=\int_{0}^{\infty}\alpha(s)\,ds=\infty.$$}

\medskip

Put
\begin{gather*}
  G(v)=c_3v^\sigma(v+c_1 \sqrt{v}), \quad g(v)= v-c_2 \sqrt{v},\quad a(t,x)=\lambda_W(t,x),
\end{gather*}
and in order to satisfy the rest of conditions which guarantee the existence of
V-bounded solution, define the family of ellipsoidal disks
 \begin{gather*}
  \mathcal{M}_t:=\{t\}\times \{x\in \mathbb L_+(t):\langle C(t)x,x\rangle \le w^+\}.
\end{gather*}
Now prove the following proposition.
\begin{lemma}\label{lem:retrquadform3} For any
$c>0$, $t_0\in \mathbb R$ there exists a retraction of the set $W^{-1}(c)$ to the
ellipsoid
\begin{gather*}
\left\{(t,x)\in \mathbb R^{1+n}:t=t_0,\quad x\in
  \mathbb L_+(t_0),\quad \langle C(t_0)x,x\rangle =c\right\}.
\end{gather*}
 \end{lemma}
\begin{proof} First we observe that for arbitrary  $t\in \mathbb R$ ³ $c>0$ there exists
a retraction of $\mathcal{N}_{t,c}:=\{x\in \mathbb R^n:\langle C(t)x,x\rangle =c\}$ to the intersection
of this set by the subspace $\mathbb L_+(t)$. In fact, one can define such a
retraction
 by a mapping $x \mapsto \theta(t,x)P_+(t)x$, provided that the scalar function
 $\theta(t,x)$ is determined from condition $\langle C_+(t)\theta(t,x)x,\theta(t,x)x\rangle =c$ for
all $x\in \mathcal{N}_{t,c}$. Since $c>0$, then $\mathcal{N}_{t,c}\cap
\mathbb L_{-}(t)=\varnothing$, and hence, $\langle C_+(t)x,x\rangle >0$ for all $x\in
\mathcal{N}_{t,c}$. Therefore
\begin{gather*}
  \theta(t,x)=\sqrt{\frac{c}{\langle C_+(t)x,x\rangle}}.
\end{gather*}

Now it remains only to show that the set  $\{t_0\}\times
\mathcal{N}_{t_0,c}=W^{-1}(c)\cap\Pi_{t_0}$ is a retract of $W^{-1}(c)$. Introduce the
operator $S(t):=\sqrt{C^2(t)}=C_+(t) -C_-(t)$, where the operators $C_{\pm}(t)$
are defined in \eqref{eq:C+C_}. Then we get
\begin{gather*}
  C(t)=S(t)(P_+(t) -P_-(t))=(P_+(t) -P_-(t))S(t).
\end{gather*}
The quadratic form  $\langle C(t)x,x\rangle $ by means of the substitution
$x=\left[\sqrt{S(t)}\right]^{-1}y$ is reduced to  $\langle (P_+(t)-P_-(t))y,y\rangle$.
Obviously, $P_+(t)-P_-(t)$ is  a symmetric orthogonal inversion operator:
$$(P_+(t)-P_-(t))^*=P_+(t)-P_-(t),\quad  (P_+(t)-P_-(t))^2=E.$$

 From the representation of projector via the Riesz formula (see, e.g.,
\cite[c.~34]{DalKre70}) it follows that the projectors $P_{\pm}(t)$  smoothly
depend on parameter. Therefore the mutually orthogonal subspaces $\mathbb L_+(t)$ and
$\mathbb L_-(t)$ have constant dimensions $n_+$,  $n_{-}$ and define smooth curves
$\gamma_+$, $\gamma_-$ in Grassmannian  manifolds $G(n,n_{+})$ and $G(n,n_{-})$
respectively. Since $G(n,n_{+})$ is a base space of a principal fiber bundle,
namely, $G(n,n_{+})=O(n)/O(n_+)\times O(n_-)$, then there exists a smooth curve
$Q(t)$ in $O(n)$, which is projected onto  $\gamma_+(t)$, the operator $Q(t_0)$
being the identity element $E$  of the group $O(n)$. Obviously,
$\mathbb L_+(t)=Q(t)\mathbb L_+(t_0)$ and, as a consequence,
\begin{gather*}
  P_{\pm}(t)=Q(t)P_{\pm}(t_0)Q^{-1}(t).
\end{gather*}

From the above reasoning it follows that the change of variables
\begin{gather*}
  x=\left[\sqrt{S(t)}\right]^{-1}Q(t)\sqrt{S(t_0)}y
\end{gather*}
reduces the quadratic form  $W(t,x):=\langle C(t)x,x\rangle $ to $W(t_0,y)=\langle C(t_0)y,y\rangle
$, and then the mapping
\begin{gather*}
  \mathbb R\times \mathbb R^{n} \mapsto\{t_0\}\times \mathbb R^n:\quad  (t,x)  \mapsto
  \left(t_0,\sqrt{S(t)}Q^{-1}(t)\left[\sqrt{S(t_0)}\right]^{-1}x\right)
\end{gather*}
define a retraction of the set $W^{-1}(c)$ to the set $W^{-1}(c)\cap \Pi_{t_0}$.
\end{proof}

Denote by $\lambda =\lambda_-^+(t)$ the minimal characteristic value of the pencil
\begin{gather*}
  P_+(t)\big[C(t)-\lambda B(t)\big]\big|_{\mathbb L_+(t)}.
\end{gather*}
Then
\begin{gather*}
  \max\{\langle B(t)x,x\rangle :x\in \mathcal{M}_t\}=\frac{w^+}{\lambda_-^+(t)},
\end{gather*}
and the condition (D) will be fulfilled once we suppose that

\medskip

\noindent (g): \quad \parbox[t]{\ppp}{the inequality
\begin{gather*}
  \limsup_{t\to -\infty}\lambda_-^+(t)>0
\end{gather*}
holds true.}

\medskip

\noindent In this case we have
\begin{gather}\label{eq:nu}
  \nu =\liminf_{t\to -\infty}\frac{w^+}{\lambda_-^+(t)}.
\end{gather}

 From the above reasoning it follows that
\begin{gather*}
  \min\{\langle C(t)x,x\rangle :x\in \mathcal{M}_t,\;\langle  B(t)x,x\rangle \ge v_0\}=
  \lambda_-^+(t)v_0.
\end{gather*}
Hence,
\begin{gather}\label{eq:tildeww0}
  \tilde{\omega}=\liminf_{t\to -\infty}\lambda_-^+(t)v_0,
  \quad \omega_0:=\liminf_{t\to -\infty}\lambda_-(t)v_0.
\end{gather}

\medskip
We have established the following result.
\begin{theorem}\label{th:applthbs}
Let the functions $V(t,x)$, $W(t,x)$ are defined by \eqref{eq:VWbilin} and the
system   \eqref{eq:nlsys} satisfies the conditions (a)--(g). Put
\begin{gather*}
  F(v):=
  \frac{1}{c_3}\int_{v_0}^{v}\frac{u-c_2\sqrt{u}}{u^{\sigma}(u+c_1\sqrt{u})}\,du
\end{gather*}
and assume that the inequality \eqref{eq:V*>} is valid where the numbers
$\nu,\;\tilde{\omega},\omega_0$ are defined by the formulae \eqref{eq:nu},
\eqref{eq:tildeww0}, and also
$\mathrm{cls}\left(V^{-1}\bigl([0,V^*)\bigr)\cap\mathcal{W}\right)\subset \Omega $. Then the system
has a V-bounded solution $x(t)$ such that
\begin{gather*}
  V(t,x(t))\le
  F^{-1}\left(\frac{v_0}{2}\Big[\sup_{s\ge t}\lambda^+(s)-\inf_{s\le t}\lambda_-(s)\Big]\right)\quad
\quad \forall   t\in \mathbb R.
\end{gather*}
\end{theorem}

To make the estimates of  V-bounded solution more efficient let us utilize
Remark~\ref{rem:4} and estimate the $F(v)$ from below. If $\sigma <1$, then for
$u\ge c_2^2$ we have
\begin{gather*}
  \frac{u-c_2\sqrt{u}}{u^{\sigma}(u+c_1\sqrt{u})}=
  \frac{1-c_2u^{-1/2}}{u^{\sigma}(1+c_1u^{-1/2})}\ge
  \frac{\sqrt{v_0}}{\sqrt{v_0}+c_1}u^{-\sigma}(1-(u/c_2^2)^{-1/2})\ge \\
   \frac{\sqrt{v_0}}{\sqrt{v_0}+c_1}(u^{-\sigma}-c_2^{1-\sigma}u^{-1/2-\sigma /2}).
\end{gather*}
Hence, in this case
\begin{gather*}
  F(v)\ge\frac{\sqrt{v_0}}{(1-\sigma)(\sqrt{v_0}+c_1)c_3}
  \left[v^{1-\sigma}-2c_2^{1-\sigma}v^{(1-\sigma)/2}-v_0^{1-\sigma}+2c_2^{1-\sigma}v_0^{(1-\sigma)/2}\right]=:F_1(v)
\end{gather*}
and
\begin{gather*}
  F_1^{-1}(z)=\left[\sqrt{\frac{(1-\sigma)(\sqrt{v_0}+c_1)c_3}{\sqrt{v_0}}z+
  \left(v_0^{(1-\sigma)/2}-c_2^{1-\sigma}\right)^{2}}+c_2^{1-\sigma}\right]^{2/(1-\sigma)}.
\end{gather*}

If $\sigma =1$, then
\begin{gather*}
  F(v)\ge \frac{\sqrt{v_0}}{(\sqrt{v_0}+c_1)c_3}
  \left[\ln v -\ln v_0 +2c_2v^{-1/2}-2c_2v_0^{-1/2}\right].
\end{gather*}
In this case we put
\begin{gather*}
  F_1(v):=\frac{\sqrt{v_0}}{(\sqrt{v_0}+c_1)c_3}
  \left[\ln v -\ln v_0 -2c_2v_0^{-1/2}\right].
\end{gather*}
Then
\begin{gather*}
  F_1^{-1}(z)= v_0\exp \left(\frac{\sqrt{v_0}+c_1}{\sqrt{v_0}}z+2c_2v_0^{-1/2}\right).
\end{gather*}

Approaching the limit as $v_0 \to  c_2^2$, we obtain the following proposition.

\begin{theorem}\label{th:applthbs2}
Let the conditions of the theorem ~\ref{th:applthbs} hold true for
$v_0=(1+\epsilon)c_2^2$ and for all sufficiently small $\epsilon >0$. Then the system
\eqref{eq:nlsys} has a solution $x(t)$ which is extendable on the entire real
axis and which for any $t\in \mathbb R$ satisfies the inequality
\begin{gather*}
  \langle B(t)x(t),x(t)\rangle \le
  \begin{cases}
    \left[C_1\sqrt{\sup_{s\ge t}\lambda^+(s)-\inf_{s\le
  t}\lambda_-(s)}+c_2^{1-\sigma}\right]^{\frac{2}{1-\sigma}}, & \quad \text{ÿêùî}\;\sigma <1, \\
    (ec_2)^2\exp \left[C_2\left(\sup_{s\ge t}\lambda^+(s)-\inf_{s\le
  t}\lambda_-(s)\right)\right], & \quad \text{ÿêùî}\;\sigma =1,
  \end{cases}
\end{gather*}
where $C_1:=\sqrt{(1-\sigma)C_2}$, $C_2:=\frac{(c_1+c_2)c_2c_3}{2}$.
\end{theorem}

Lastly, we prove the following uniqueness theorem.

\begin{theorem}\label{th:unbs2} Let for the system~\eqref{eq:nlsys} the
following conditions hold true:

 1) there exists a domain $\hat{\Omega}\subseteq \Omega $ such that
 \begin{gather*}
  f(t,x)-f(t,y)=\hat{A}(t,x,y)(x-y)\quad \forall (t,x),(t,y)\in \hat{\Omega},
\end{gather*}
where $\hat{A}(t,x,y)\in \mathrm{Hom}(\mathbb R^n)$;

2) there exists a smooth family of operators $\{\hat{C}(t)\}_{t\in \mathbb R}$ in $\mathbb R^n$
and a function $\hat{\beta}(\cdot)\in C(\mathbb R \mapsto(0,\infty))$ such that the minimal
characteristic value $\hat{\lambda}(t,x,y)$ of the pencil
$$\hat{C}(t)\hat{A}(t,x,y)+\hat{A}^*(t,x,y)\hat{C}(t)+\frac{d}{dt}\hat{C}(t)-\lambda B(t)$$
satisfies the inequality  $$\hat{\lambda}(t,x,y)\ge \hat{\beta}(t)\quad \forall
(t,x),\;(t,y)\in \hat{\Omega};$$

3) the maximal by absolute value characteristic value $\hat{\Lambda}(t)$   of the
pencil  $\hat{C}(t)-\lambda B(t)$ satisfies the equality
\begin{gather*}
  \limsup_{t\to \pm \infty}
  \frac{1}{\left|\hat{\Lambda}(t)\right|}\left|\int_{0}^{t}\frac{\beta(s)}{\hat{\Lambda}(s)}\,ds\right|=\infty.
\end{gather*}

Then the system \eqref{eq:nlsys} has at most one solution $x(t)$ defined on the
entire axis, with the graph belonging to $\hat{\Omega}$ and such that
 $\sup_{t\in \mathbb R}\langle B(t)x(t),x(t)\rangle <\infty $.
\end{theorem}
\begin{proof} Having observed that the theorem's conditions ensure the fulfillment
 of the inequalities
\begin{gather*}
|\langle \hat{C}(t)(x-y),x-y\rangle|\le \left|\hat{\Lambda}(t)\right|\langle B(t)(x-y),x-y\rangle,\\
  \langle \big(2\hat{C}(t)\hat{A}(t,x,y)+\tfrac{d}{dt}\hat{C}(t)\big)(x-y),x-y\rangle \ge
  \hat{\beta}(t)\langle B(t)(x-y),x-y\rangle,
\end{gather*}
it is sufficient to apply the theorem~\ref{th:unbs} for the case where
$U(t,x)=\langle \hat{C}(t)x,x\rangle $, $V(t,x)=\langle B(t),x,x\rangle $.
\end{proof}

{\bf Conclusions}.

The technique applied in this paper for studying the essentially nonlinear
nonauto\-nom\-ous systems  by means of a pair of auxiliary functions allows us
to generalize a number of earlier known results concerning the  questions of
existence and uniqueness of bounded and proper solutions. In the case where the
estimating function is a quadratic form with varying matrix the
theorem~\ref{th:applthbs2} can be efficiently applied to establish asymptotic
estimates of solutions for $t\to \infty $.


\end{document}